\pdfoutput=1
\RequirePackage{ifpdf}
\ifpdf 
\documentclass[pdftex]{sigma}
\else
\documentclass{sigma}
\fi

\usepackage{tikz}

\numberwithin{equation}{section}

\newtheorem{Theorem}{Theorem}[section]
\newtheorem{Corollary}[Theorem]{Corollary}
\newtheorem{Lemma}[Theorem]{Lemma}
 { \theoremstyle{definition}
\newtheorem{Remark}[Theorem]{Remark} }

\begin{document}

\allowdisplaybreaks

\renewcommand{\thefootnote}{$\star$}

\newcommand{\arXivNumber}{1603.04328}

\renewcommand{\PaperNumber}{093}

\FirstPageHeading

\ShortArticleName{Precise Deviations Results for the Maxima of Some Determinantal Point Processes}

\ArticleName{Precise Deviations Results for the Maxima of Some\\ Determinantal Point Processes: the Upper Tail\footnote{This paper is a~contribution to the Special Issue on Asymptotics and Universality in Random Matrices, Random Growth Processes, Integrable Systems and Statistical Physics in honor of Percy Deift and Craig Tracy. The full collection is available at \href{http://www.emis.de/journals/SIGMA/Deift-Tracy.html}{http://www.emis.de/journals/SIGMA/Deift-Tracy.html}}}

\Author{Peter EICHELSBACHER~$^\dag$, Thomas KRIECHERBAUER~$^\ddag$ and Katharina SCH\"ULER~$^\ddag$}

\AuthorNameForHeading{P.~Eichelsbacher, T.~Kriecherbauer and K.~Sch\"uler}

\Address{$^\dag$~Fakult\"at f\"ur Mathematik, Ruhr-Universit\"at Bochum, 44780 Bochum, Germany}
\EmailD{\href{mailto:peter.eichelsbacher@rub.de}{peter.eichelsbacher@rub.de}}
\URLaddressD{\url{http://www.ruhr-uni-bochum.de/ffm/Lehrstuehle/stochastik/}}

\Address{$^\ddag$~Mathematisches Institut, Universit\"at Bayreuth, 95440 Bayreuth, Germany}
\EmailD{\href{mailto:thomas.kriecherbauer@uni-bayreuth.de}{thomas.kriecherbauer@uni-bayreuth.de}, \href{mailto:katharina.schueler@uni-bayreuth.de}{katharina.schueler@uni-bayreuth.de}}
\URLaddressD{\url{http://www.diffgleichg.uni-bayreuth.de/en/}}

\ArticleDates{Received May 31, 2016, in f\/inal form September 11, 2016; Published online September 21, 2016}

\Abstract{We prove precise deviations results in the sense of Cram\'er and Petrov for the upper tail of the distribution of the maximal value for a special class of determinantal point processes that play an important role in random matrix theory. Here we cover all three regimes of moderate, large and superlarge deviations for which we determine the leading order description of the tail probabilities. As a corollary of our results we identify the region within the regime of moderate deviations for which the limiting Tracy--Widom law still predicts the correct leading order behavior. Our proofs use that the determinantal point process is given by the Christof\/fel--Darboux kernel for an associated family of orthogonal polynomials. The necessary asymptotic information on this kernel has mostly been obtained in [Kriecherbauer T., Schubert K., Sch{\"u}ler K., Venker M., \textit{Markov Process. Related Fields} \textbf{21} (2015), 639--694]. In the superlarge regime these results of do not suf\/f\/ice and we put stronger assumptions on the point processes. The results of the present paper and the relevant parts of [Kriecherbauer T., Schubert K., Sch{\"u}ler K., Venker M., \textit{Markov Process. Related Fields} \textbf{21} (2015), 639--694] have been proved in the dissertation [Sch{\"u}ler K., Ph.D.~Thesis, Universit{\"a}t Bayreuth, 2015].}


\Classification{60F10; 60B20; 35Q15; 42C05}

\begin{flushright}
\begin{minipage}{100mm}
\it Dedicated to Percy Deift and Craig Tracy on the occasion of their 70'th birthdays with deep admiration for their ground breaking work and great leadership!
\end{minipage}
\end{flushright}

\renewcommand{\thefootnote}{\arabic{footnote}}
\setcounter{footnote}{0}

\section{Introduction}

\subsection{Model and general assumptions}

In this paper we consider a class of determinantal point processes that is prominent in random matrix theory. There a well studied ensemble type consists of probability measures on $N \times N$ Hermitian matrices that are invariant under unitary conjugation (unitary invariant ensembles) and for which the joint distribution of the vector $\lambda \in \mathbb{R}^N$ of eigenvalues has a density of the form
\begin{gather}\label{defensemble}
P_{N,V}(\lambda) = Z_{N,V}^{-1} \prod_{1\leq j<k\leq N} (\lambda_k-\lambda_j)^2 \prod_{r=1}^N e^{-N V(\lambda_r)}.
\end{gather}
The function $V\colon \mathbb{R} \to \mathbb{R}$ should be viewed as a parameter of the model and is supposed to have suf\/f\/icient growth at $\pm \infty$ such that the measure can be normalized by a constant $Z_{N, V}$. Note that choosing $V$ to be a quadratic function leads to the classic Gaussian unitary ensemble~(GUE). We refer the interested reader to \cite{handbookRM, bookAGZ, Deift, bookForrester, bookMehta, bookPasturS} for recent monographs on random matrix theory.

The determinantal nature of the point process on $\mathbb{R}$ induced by the probability measure ${\mathrm d}\mathbb P_{N, V}(\lambda) =P_{N, V}(\lambda)\,{\mathrm d}\lambda $ on $N$-point conf\/igurations is due to the square of the Vandermonde determinant appearing in \eqref{defensemble}.
In fact, there exist functions $K_{N, V}\colon \mathbb{R}^2 \to \mathbb{R}$ such that all marginal densities (also called correlation functions) can be expressed as determinants of the $N \times N$ matrix $\mathcal{K}_{N, V}(\lambda) :=(K_{N, V}(\lambda_j, \lambda_k))_{1\leq j, k\leq N}$ and of its principal minors, e.g., $P_{N,V}(\lambda)= \frac{1}{N!}\det [\mathcal{K}_{N, V}(\lambda)]$ (see, e.g., \cite[Section~2.3]{Sosh}, see also \cite[Section~4.2]{bookAGZ}, \cite[Chapters~4 and~11]{handbookRM} and references therein). Moreover, the kernels can be expressed in terms of orthogonal polynomials w.r.t.\ the measure $e^{-NV(x)}{\mathrm d}x$ on~$\mathbb{R}$. More precisely, denote by $\big(p_j^{N,V}\big)_j$ the uniquely def\/ined sequence of polynomials that satisf\/ies
\begin{gather*}
\int_{\mathbb{R}} p_j^{N,V}(x) p_k^{N,V}(x) e^{-NV(x)} \, {\mathrm d}x = \delta_{j, k}, \qquad \text{for all} \ \ j, k \geq 0,
\end{gather*}
and where $\deg\big(p_j^{N,V}\big)=j$ with a positive leading coef\/f\/icient. Then $K_{N, V}$ is given by the corresponding Christof\/fel--Darboux kernel \cite[Section~6.3]{handbookRM}
\begin{gather}\label{defCDK}
K_{N, V}(x,y) := \sum_{j=0}^{N-1} p_j^{N,V} (x) p_j^{N,V}\!(y) e^{-\frac{N}{2}(V(x)+V(y))}, \qquad \text{for} \ \ x, y \in \mathbb{R}.
\end{gather}

We are interested in deviations results for the distribution of the largest component $\lambda_{\max} := \max\{\lambda_1, \ldots, \lambda_N\}$ of $\lambda$ in the limit as $N \to \infty$. In order to be def\/inite in our subsequent discussion we now introduce the general assumptions~(\textbf{GA}) on the functions~$V$ that will be required throughout this paper. These are certainly not the most general conditions for our results to hold true but they reduce the technicalities in the proofs to a minimum.

A function $V$ is said to satisfy (\textbf{GA}) if (1)--(3) hold:
\begin{enumerate}\itemsep=0pt
\item[(1)] $V\colon \mathbb{R}\to \mathbb{R}$ is real analytic;
\item[(2)] $V'$ is strictly monotonically increasing (convexity assumption);
\item[(3)] $\lim\limits_{| x|\to\infty}V(x)=\infty$.
\end{enumerate}
\vspace{3pt}

Note that conditions (2) and (3) imply at least linear growth of $V(x)$ for $|x|\to\infty$ that suf\/f\/ices to ensure the integrability of $P_{N,V}$.

\subsection{Equilibrium measure and upper tail rate function}

One important ingredient in the analysis of the probability measure $\mathbb P_{N, V}$ on $\mathbb{R}^N$ is the functional
\begin{gather}\label{functional_I}
	I_V(\mu):=\int_{\mathbb{R}} \int_{\mathbb{R}} \log |x-y|^{-1} \,{\mathrm d}\mu(x)\,{\mathrm d}\mu(y) + \int_{\mathbb{R}} V(x)\,{\mathrm d}\mu(x)
\end{gather}
def\/ined on $\mathcal{M}_1(\mathbb{R}) := \{ \mu \colon \mu$ Borel measure on $\mathbb{R}$ with $\mu(\mathbb{R})=1 \}$. The connection can be explained heuristically as follows: Associate any $\lambda=(\lambda_1,\ldots,\lambda_N)\in\mathbb{R}^N$ with its normalized counting measure $\mu_{\lambda}:=\frac{1}{N}\sum\limits_{j=1}^N \delta_{\lambda_j}$. Then $P_{N,V}(\lambda)=Z_{N,V}^{-1} \exp \big({-}N^2 \tilde{I}_V(\mu_{\lambda})\big)$ where $\tilde{I}_V$ dif\/fers from $I_V$ by excluding the diagonal from the domain of integration in the f\/irst summand of~\eqref{functional_I}. It is now plausible that $\mathbb P_{N, V}$ concentrates around those vectors $\lambda$ for which $I_V(\mu_{\lambda})$ is close to the inf\/imum of $I_V(\mu)$ where $\mu$ ranges over $\mathcal{M}_1(\mathbb{R})$. In fact, this observation can be used to derive a large deviations principle for both the counting measure $\mu_{\lambda}$ and for $\lambda_{\max}$ under $P_{N,V}$ for $V$ satisfying~(\textbf{GA}) (see, e.g., \cite[Section~2.6]{bookAGZ} and references therein). For a def\/inition of a~large deviations principle, see \cite{bookDemboZeitouni}. However, this is not the approach of the present paper in which the analysis is based on the determinantal representation of $P_{N,V}$.

Let us summarize some well known facts about the minimization of the functional $I_V$, see, e.g., \cite[Chapter~6]{Deift}, \cite[Chapter~11]{bookPasturS} and references therein, see also \cite[Chapter~2]{Diss} for a~derivation of the facts relevant in the present paper. For a~large class of functions~$V$, including those satisfying~(\textbf{GA}), the functional $I_V$ has a unique minimizer $\mu_V$ that is called the equilibrium measure with respect to the external f\/ield~$V$. The equilibrium measure $\mu_V$ has compact support and we denote by $b_V$ the maximum of its support. From a heuristic point of view we expect~$\lambda_{\max}$ to f\/luctuate around~$b_V$. In order to describe these f\/luctuations it is known that the gradient~$\mathcal{L}_V$ of~$I_V$ at~$\mu_V$ comes into play. It is given by
\begin{gather}\label{aux_eta}
\mathcal{L}_V(x) = 2 \int_{\mathbb{R}} \log |x-y|^{-1}\, {\mathrm d}\mu_V(y) + V(x), \qquad x \in \mathbb{R}.
\end{gather}
The Euler--Lagrange equations for the above variational problem imply that there exists a real number~$l_V$ such that $\mathcal{L}_V$ equals $l_V$ on the support of $\mu_V$ and is $\geq l_V$ elsewhere. Hence the function
\begin{gather}\label{def1_eta}
\eta_V := \mathcal{L}_V - l_V
\end{gather}
is non-negative and vanishes identically on the support of the equilibrium measure. Observe that~$\eta_V$ coincides with the rate function of the large deviations principle for the upper tail of~$\lambda_{\max}$, see \cite[Section~2.6, Theorem~2.6.6]{bookAGZ}.

It is a remarkable and useful fact that in the case of strictly convex and suf\/f\/iciently smooth functions~$V$ (e.g.,~$C^3$ will do) there is an almost explicit representation for $\eta_V$: There exist~-- and this is the implicit part~-- unique reals $a_V < b_V$, called Mhaskar--Rakhmanov--Saf\/f numbers, that are uniquely def\/ined by the two equations
\begin{gather}\label{mrs_eq}
\int_a^b \frac{V'(t)}{\sqrt{(b-t)(t-a)}}\, {\mathrm d}t = 0, \qquad \int_a^b \frac{tV'(t)}{\sqrt{(b-t)(t-a)}}\, {\mathrm d}t = 2\pi.
\end{gather}
As it turns out the support of $\mu_V$ equals the interval $[a_V, b_V]$. Set
\begin{gather}\label{def_G}
 	G_V\colon \ \mathbb{R}\to\mathbb{R}, \qquad G_V(x):=\frac{1}{\pi} \int_{a_V}^{b_V} \frac{V'(x)-V'(t)}{x-t} \frac{1}{\sqrt{(b_V-t)(t-a_V)}}\,{\mathrm d}t.
\end{gather}
Observe that $G_V > 0$ by condition (2) of~(\textbf{GA}). To the right of the support of $\mu_V$ we have
\begin{gather}\label{def_eta}
	\eta_V(x) = \int_{b_V}^x \sqrt{(u-b_V)(u-a_V)} G_V(u)\,{\mathrm d}u \qquad \text{for} \ \ x > b_V.
\end{gather}

This implies in particular for $x > b_V$ that $\eta_V$ is of order $(x-b_V)^{3/2}$ near $b_V$. More precisely, for small positive values of $x-b_V$ the following holds
\begin{gather}\label{def_gamma}
\eta_V(x) = \frac{4}{3} [\gamma_V(x-b_V) ]^{3/2} ( 1 + \mathcal{O}(x-b_V) ) \qquad \text{with} \quad
	\gamma_V:= \big[\tfrac{1}{2} \sqrt{b_V-a_V} G_V(b_V)\big]^{2/3}.\!\!\!
\end{gather}
Note that the prefactor $\frac{4}{3}$ has no signif\/icance other than the standard convention that \mbox{$\gamma_V=1$} for $V(x) = x^2/2$. Secondly, we remind the reader that the equilibrium measure can also be expressed in terms of the just def\/ined quantities. Indeed, on its support~$\mu_V$ is given by ${\mathrm d} \mu_V (x) = \frac{1}{2\pi}\sqrt{(b_V-x)(x-a_V)} G_V(x)\, {\mathrm d} x$.

\subsection[Fluctuations: Tracy--Widom law, large and moderate deviations principles]{Fluctuations: Tracy--Widom law, large and moderate\\ deviations principles}

After the brief review of the equilibrium measure we are now ready to discuss the f\/luctuations of $\lambda_{\max}$ around $b_V$. They are of order $N^{-2/3}$ and, appropriately rescaled (here we also need the just def\/ined~$\gamma_V$), they are asymptotically described by the celebrated ($\beta=2$) Tracy--Widom distribution~$F_{\rm TW}$~\cite{TracyWidom94}, i.e.,
\begin{gather}\label{tw}
	\lim_{N\to\infty} \mathbb{P}_{N,V}\left( \frac{\lambda_{\max}-b_V}{(\gamma_V N^{2/3})^{-1}} \leq s\right)	= F_{\rm TW}(s)
\end{gather}
uniformly for $s\in \mathbb{R}$ (see, e.g., \cite[Section~6.3]{handbookRM}, \cite[Chapter~6]{DeiftGioev}, \cite[Section~13.1]{bookPasturS}
and references therein). This result can be viewed as an analogue of the central limit theorem for the arithmetic mean of $N$ independent and identically distributed random variables. Note that the role of the normal distribution is taken by the Tracy--Widom distribution and that the power law of the f\/luctuations has changed from~$N^{-1/2}$ to~$N^{-2/3}$. As it is the case for the classical CLT it is natural to ask for corresponding deviations results. Roughly speaking this means to describe how fast the tail probabilities tend to zero if~$s$ is not kept f\/ixed as in~\eqref{tw} but is allowed to grow with~$N$. In this paper we are only concerned with the upper tail.

Before we formulate our precise deviations results for $\lambda_{\max}$ in Theorems~\ref{thm_O} and~\ref{theorem_sl} we begin by stating our results in a weaker but possibly more familiar form that is related to the large deviations principles introduced by Varadhan (see, e.g.,~\cite{bookDemboZeitouni}). We will show in Corollary~\ref{coro_univ} how to derive these from our main results.

Recall the def\/inition of $\eta_V$ in \eqref{def1_eta} (see also~\eqref{def_eta}). Then we have for $t > b_V$:
\begin{gather}\label{ldp}
\frac{1}{N} \log \mathbb{P}_{N,V}( \lambda_{\max}>t) = - \eta_V(t) -\frac{\log N}{N} + \mathcal{O}\left(\frac{1}{N}\right),
\end{gather}
where the $\mathcal{O}$-term is uniform for $t$ in compact subsets of $(b_V, \infty)$. Formula \eqref{ldp} implies in particular a large deviations principle for $\lambda_{\max}$ with speed $N$ and rate function $J_V(t) = \eta_V(t)$ if $t \geq b_V$ and $J_V(t) = \infty$ otherwise. Indeed, we obtain that $\limsup_N \frac 1N \log \mathbb{P}_{N,V}( \lambda_{\max} \leq t) = - \infty$ for any $t < b_V$ applying the large deviations principle for the empirical measure of the eigenvalues, see \cite[equation~(2.6.30)]{bookAGZ}. Furthermore, one can remark that together with the assertion $\lim_N \frac 1N \log \mathbb{P}_{N,V}( \lambda_{\max} \geq t) = - \eta_V(t)$ for all $t\geq b_V$, Theorem~4.1.11 in~\cite{bookDemboZeitouni} allows us to derive a large deviations principle from the limiting behavior of probabilities for a basis of topology, see also~\cite{ERSY}. Observe, that for any $V$ satisfying condition (\textbf{GA}), growth-condition~(2.6.2) and Assumption~2.6.5 in \cite[Theorem~2.6.6]{bookAGZ} are fulf\/illed. The latter follows from Lemmas~4.5 and~4.6 in~\cite{JohT} (for discrete Coulomb gases, but the proof for a continuous Coulomb gas is essentially the same). Hence~\eqref{ldp} reproves Theorem~2.6.6 in~\cite{bookAGZ} under stronger assumptions but provides more information on the higher order terms.

Large deviations principles for extremal values have already been proved in a much more general setting of mean f\/ield models with Coulomb gas interactions that do not necessarily possess the structure of determinantal point processes, (see \cite{BDG}, \cite[Section~2.6.2]{bookAGZ}, \cite[Section~4]{JohT}, \cite{BoGaGu, CE, Feral}). Note that these results do not provide rates of convergence as presented in~\eqref{ldp}. Another class of repulsive particle systems that is not determinantal but can be expressed as an average of determinantal ones by a~stochastic linearization procedure has been introduced in~\cite{GoVe}. For such ensembles a result comparable to~\eqref{ldp} has been obtained in~\cite{KrVe}. Recently in~\cite{Fanny} the author proves a large deviations principle for the largest eigenvalue of Wigner matrices without Gaussian tails, namely such that the distribution tails $P(|X_{1,1}| >t)$ and $P(|X_{1,2}| >t)$ behave like $e^{-b t^{\alpha}}$ and $e^{-at^{\alpha}}$ respectively for some $a,b \in (0, \infty)$ and $\alpha \in (0,2)$. The large deviations principle is of speed~$N^{\alpha/2}$ and with an explicit rate function depending only on the tail distributions of the~$X_{i,j}$.

We turn to the regime of moderate deviations. Theorem~\ref{thm_O} below implies for any $\alpha \in \big(0, \frac{2}{3}\big)$ moderate deviations principles for the rescaled random variable $\widetilde{\lambda}_{\max}/N^{\alpha}$ with
\begin{gather*}
\widetilde{\lambda}_{\max} := (\lambda_{\max}-b_V)\gamma_V N^{2/3}
\end{gather*}
as appearing in \eqref{tw}. Here the speed is $N^{\frac{3}{2}\alpha}$ and the rate function is $J(s) := \frac{4}{3} s^{3/2}$. This can be seen from the following corollary of Theorem~\ref{thm_O}:
\begin{gather}\label{mdp}
\frac{1}{s^{3/2}} \log \mathbb{P}_{N,V}\big( \widetilde{\lambda}_{\max} > s\big)
	 = -\frac{4}{3}-\frac{\log (16\pi s^{3/2})}{s^{3/2}} + \mathcal{O}\left( \frac{s}{N^{2/3}}\right) + \mathcal{O}\left( \frac{1}{s^3}\right),
\end{gather}
where the $\mathcal{O}$-terms are uniform for $s \in \big[1, N^{2/3}\big]$, thus connecting the Tracy--Widom regime with the regime of large deviations.

In the regime of moderate deviations less is known. Upper and lower bounds on the left hand side of \eqref{mdp} of the correct order were shown in~\cite{LedouxRider} for Hermitian $\beta$-ensembles that are determinantal for $\beta = 2$ only and agree in this case with the Gaussian unitary ensemble. A~result of the form~\eqref{mdp} has been proved in~\cite{KrVe} for the class of repulsive particle systems introduced in~\cite{GoVe}. Finally we refer the reader to the moderate deviations results in~\cite{hanna1, hanna2, hanna3} on certain statistics of eigenvalues for Wigner matrices and to the moderate deviations results in~\cite{BDMMZ, Merkl1, Merkl2} on combinatorial models that are closely related to random matrix theory.

\subsection{Precise deviations results I: moderate and large deviations}

As already mentioned, all the above results follow from our precise deviations results in the sense of Cram\'er and Petrov. Here the goal is to identify the leading order description of \mbox{$\mathbb{P}_{N,V}( \lambda_{\max}>t)$} resp.~$\mathbb{P}_{N,V}(\widetilde{\lambda}_{\max} > s)$, i.e., to identify functions $\mathcal{F}_{N,V}$ resp.~$\widetilde{\mathcal{F}}_{N,V}$ such that $\mathbb{P}_{N,V}(\lambda_{\max}>t)/\mathcal{F}_{N,V}(t)$ resp.~$\mathbb{P}_{N,V}( \widetilde{\lambda}_{\max} > s)/\widetilde{\mathcal{F}}_{N,V}(s)$ tend to $1$ as $N \to \infty$ in suitable subsets of the~$(t,N)$ resp.~$(s, N)$ plane. For example, for any bounded subset $B \subset [0, \infty)$ we learn from the Tracy--Widom law~\eqref{tw} that for $(s, N) \in B \times \mathbb N$ we have
\begin{gather}\label{precisetw}
\lim_{N \to \infty} \frac{\mathbb{P}_{N,V}\big( \widetilde{\lambda}_{\max} > s\big)}{1-F_{\rm TW}(s)} =1.
\end{gather}
By this we mean that for any sequence $(s_N, N)_N$ in $B \times \mathbb N$ relation~\eqref{precisetw} holds with $s$ being replaced by~$s_N$. Observe that due to the fast decay of $1-F_{\rm TW}(s)$ for $s \to \infty$ (see~\eqref{astw} below) even optimal control on the rate of convergence in~\eqref{tw} would only allow to extend this result to values of $s$ that grow at most of order $(\log N)^{2/3}$ with~$N$. In this paper we are able to show that, in fact, \eqref{precisetw} holds true for $s = o \big( N^{4/15}\big)$ but generically (in~$V$) in no larger domain.

Our f\/irst main result allows us to obtain simultaneously the leading order description of the upper tail probabilities $\mathbb{P}_{N,V}( \lambda_{\max}>t)$ resp.~$\mathbb{P}_{N,V}( \widetilde{\lambda}_{\max} > s)$ in the regimes of large resp.\ moderate deviations. To state it conveniently, we introduce the function
\begin{gather}\label{def_FNV}
\mathcal{F}_{N,V}(t) := \frac{b_V-a_V}{8\pi} \frac{e^{-N\eta_V(t)}}{N(t-b_V)(t-a_V)\eta_V'(t)} \qquad \text{for} \ \ t > b_V.
\end{gather}

\begin{Theorem}\label{thm_O} Assume that $V$ satisfies $({\bf GA})$ and recall the notation introduced above. Then the upper tail probability satisfies for all $t>b_V$ the relation
\begin{gather}\label{formel_O}
	\mathbb{P}_{N,V}( \lambda_{\max}>t) = \mathcal{F}_{N,V}(t) \left( 1+\mathcal{O}\left( \frac{1}{N(t-b_V)^{3/2}}\right) \right),
\end{gather}
with a uniform $\mathcal{O}$-term for $t$ in bounded subsets of $(b_V,\infty)$.
\end{Theorem}

Observe that Theorem \ref{thm_O} immediately implies \eqref{ldp} with the uniformity of the $\mathcal{O}$-term claimed there. We would like to point out that for \eqref{formel_O} the uniformity of the $\mathcal{O}$-term is stated not only for compact but for bounded subsets of $(b_V,\infty)$ extending the region of validity up to~$b_V$. Note, however, that there exists a positive number $C_V$ depending on the constant in the $\mathcal{O}$-term such that for $0<t-b_V\leq C_V N^{-2/3}$ statement \eqref{formel_O} already follows from the boundedness of $\mathbb{P}_{N,V}( \lambda_{\max}>t)/\mathcal{F}_{N,V}(t)$ which is easy to derive using \eqref{def_G}--\eqref{def_gamma}. This lack of informative value of~\eqref{formel_O} for these values of $t$ is no problem since they belong to the Tracy--Widom regime and~\eqref{precisetw} with $[0, \gamma_V C_V] \subset B$ f\/ills the gap.

Next we turn to the regime of moderate deviations $N^{-2/3} \ll t \ll 1$, where by \eqref{formel_O} we have $\mathbb{P}_{N,V}( \lambda_{\max}>t)/\mathcal{F}_{N,V}(t) \to 1$ as $N \to \infty$. It is instructive to translate this result to the rescaled variable $\widetilde{\lambda}_{\max}$. Since
\begin{gather}\label{as0}
\mathbb{P}_{N,V}\big( \widetilde{\lambda}_{\max} > s\big) = \mathbb{P}_{N,V} ( \lambda_{\max}>t(s) ) \qquad \text{with} \quad t(s): = b_V + \frac{s}{\gamma_V N^{2/3}}
\end{gather}
we only need to evaluate $\mathcal{F}_{N,V}(t(s))$. Using again \eqref{def_G}--\eqref{def_gamma} and the assumed real analyticity of~$V$ one obtains for positive $s = o\big( N^{2/3}\big)$, i.e., in particular in the regime of moderate deviations, that
\begin{gather}\label{as1}
\mathcal{F}_{N,V}(t(s)) = \frac{e^{-N \eta_V (t(s))}}{16 \pi s^{3/2}} \left[ 1 + \mathcal{O} \left( \frac{s}{N^{2/3}}
\right) \right] \qquad \text{and}\\
\label{as2}
N \eta_V (t(s)) = \frac{4}{3}s^{3/2} + \sum_{j=1}^{\infty} d_{j, V} \frac{s^{j + \frac{3}{2}}}{N^{\frac{2}{3} j}} =
 \frac{4}{3}s^{3/2} + \mathcal{O} \left( \frac{s^{5/2}}{N^{2/3}} \right)
\end{gather}
for some sequence $(d_{j, V})_{j \geq 1}$ of real numbers depending on $V$.

From these formulas \eqref{mdp} is immediate, at least for $s \in \big[C, c N^{2/3}\big]$ where the positive numbers $c$, $C$ depend on the constants in the $\mathcal{O}$-terms of \eqref{formel_O} and \eqref{as1}. To extend~\eqref{mdp} to all of $\big[1, N^{2/3}\big]$ one may use the monotonicty of $\mathbb{P}_{N,V}( \widetilde{\lambda}_{\max} > s)$ in $s$ for the lower end and for the upper end one shows that $s^{3/2} e^{N \eta_V (t(s))} \mathcal{F}_{N,V}(t(s))$ is bounded away from~$0$.

We return to the leading order description for $\mathbb{P}_{N,V}( \widetilde{\lambda}_{\max} > s)$ in the regime of moderate deviations. The f\/irst observation is that combining \eqref{formel_O}--\eqref{as2} leads to a series representation for the upper tail that is the exact analogue to the Cram\'er series for sums of independent variables~\cite{Cramer}, see also \cite[Section~5.8]{bookPetrov}. Secondly, in order to determine the leading order we only need to keep those terms of the series in \eqref{as2} that do not vanish as $N$ becomes large. A~computation shows that for any $k \in \mathbb N_{0}$ and positive~$s$ we have in the limit $N \to \infty$
\begin{gather}\label{as3}
N \eta_V (t(s)) = \widetilde{\eta}_{V, k}(s, N) + o (1) \qquad \text{for} \ \ s =o(N^{\alpha_k}) ,
\end{gather}
where
\begin{gather*}
\alpha_k := \frac{2}{3} - \frac{2}{2k+5} \qquad \text{and} \qquad \widetilde{\eta}_{V, k}(s, N) := \frac{4}{3}s^{3/2} + \sum_{j=1}^{k} d_{j, V} \frac{s^{j + \frac{3}{2}}}{N^{\frac{2}{3} j}} .
\end{gather*}
Note that $\widetilde{\eta}_{V, 0}(s, N) = \frac{4}{3} s^{3/2}$ does not depend on $N$ and is just the rate function $J$ introduced above~\eqref{mdp}. The results of our discussion are summarized in statements a) and b) of the following

\begin{Corollary}\label{coro_univ}
Assume that $V$ satisfies $({\bf GA})$ and recall the notation introduced above.
\begin{enumerate}\itemsep=0pt
\item[$a)$] {\rm Deviations principles (large and moderate).} Relations~\eqref{ldp} and \eqref{mdp} hold with the uniformity of the $\mathcal{O}$-terms stated there.

\item[$b)$] {\rm Precise deviations (large and moderate).}
\begin{itemize}\itemsep=0pt
\item[$i)$] For any $q \in (0, \infty)$ and any sequence of positive reals $(p_N)_N$ satisfying $p_N < q$ and $\lim\limits_{N\to \infty} p_N N^{2/3} = \infty$ we have
\begin{gather*}
\mathbb{P}_{N,V}( \lambda_{\max}>t) = \mathcal{F}_{N,V}(t) (1 + o(1))
\end{gather*}
uniformly in $t \in [b_V+p_N, b_V+q]$ as $N\to \infty$.
\item[$ii)$] For any $k \in \mathbb N_0$ and any sequences of positive reals $(\widetilde{p}_N)_N$, $(\widetilde{q}_N)_N$ satisfying $\widetilde{p}_N < \widetilde{q}_N$ and $\widetilde{p}_N \to \infty$, $\widetilde{q}_N N^{-\alpha_k} \to 0$ for $N \to \infty$, we have
\begin{gather*}
\mathbb{P}_{N,V}( \widetilde{\lambda}_{\max} > s) = \frac{\exp [-\widetilde{\eta}_{V, k}(s, N)]}{16 \pi s^{3/2}}(1 + o(1))
\end{gather*}
uniformly in $s \in [\widetilde{p}_N, \widetilde{q}_N]$ as $N\to \infty$.
\end{itemize} \vspace{7pt}

\item[$c)$] {\rm Range of validity of the Tracy--Widom law.}
For any sequence of positive reals $(\widetilde{q}_N)_N$ satisfying $\lim\limits_{N\to \infty} \widetilde{q}_N N^{-4/15} = 0$ we have
\begin{gather*}
\mathbb{P}_{N,V}\big( \widetilde{\lambda}_{\max} > s\big) = (1 - F_{\rm TW}(s)) (1 + o(1))
\end{gather*}
uniformly in $s \in [0, \widetilde{q}_N]$ as $N\to \infty$.
\end{enumerate}
\end{Corollary}

\begin{Remark}
The result in (c) can be viewed as an analogue of the Cram\'er--Petrov result for the arithmetic mean of $N$ independent and identically distributed random variables $(X_i)_i$ with zero mean and variance~1, see \cite[Theo\-rem~5.23]{bookPetrov}. Here for any sequence of positive reals $(a_N)_N$ satisfying $\lim\limits_{N \to \infty} a_N N^{-1/6} =0$ one has
\begin{gather*}
P\left( N^{-1/2} \sum_{i=1}^N X_i > s\right) = (1- \Phi(s))(1 +o(1))
\end{gather*}
uniformly in $s \in [0, a_N]$ as $N \to \infty$, where $\Phi(t)$ is the distribution function of a standard normal distributed random variable.
\end{Remark}

\begin{proof} We are only left to show statement c). Observe that $\alpha_0 = 2/3-2/5 = 4/15$. Therefore it follows from \eqref{formel_O}--\eqref{as3} that for all $s \in (0, \widetilde{q}_N]$:
\begin{gather*}
\mathbb{P}_{N,V}( \widetilde{\lambda}_{\max} > s) = \frac{e^{- \frac{4}{3}s^{3/2}}}{16 \pi s^{3/2}}
\left[ 1 + \mathcal{O} \left( \frac{s^{5/2}}{N^{2/3}} \right) + \mathcal{O} \left( \frac{1}{s^{3/2}} \right)\right] .
\end{gather*}
Using in addition the asymptotics of the Tracy--Widom distribution (see, e.g., \cite[equa\-tions~(1) and~(25)]{BBDF}, cf.~\cite[Chapter~9]{handbookRM} and references therein)
\begin{gather}\label{astw}
1-F_{\rm TW}(s)	=\frac{e^{-\frac{4}{3}s^{3/2}}}{16\pi s^{3/2}} \left[ 1 + \mathcal{O} \left( \frac{1}{s^{3/2}} \right)\right],
\end{gather}
for $s \to \infty$ and \eqref{precisetw} to deal with the $\mathcal{O} \big(s^{-3/2}\big)$-term, the claim follows.
\end{proof}

To the best of our knowledge there are three deviations results in the literature of comparable precision for models that have the Tracy--Widom distribution as their limit law. The f\/irst two are concerned with the upper tail in the moderate regime: Firstly, in~\cite{Merkl1} the leading order description is given for the length of the longest increasing subsequence of a random permutation. Secondly, for the largest particles from ensembles that were introduced in~\cite{GoVe} precise deviations were proved in \cite{KrVe} with slightly worse $\mathcal{O}$-terms that are due to the averaging procedure that is used. The third result is contained in \cite[see equation~(162)]{DeItKra} and deals with the lower tail of the distribution of the largest eigenvalue of the Laguerre Unitary Ensemble in the regime of moderate and large deviations.

\begin{Remark}
An important topic of random matrix theory is the question of universality. A~good example for a universality result is~\eqref{tw}. After an appropriate linear rescaling that involves only two $V$-dependent numbers $b_V$ and~$\gamma_V$, the distribution of the largest eigenvalue tends in the limit $N \to \infty$ to the Tracy--Widom distribution that has no $V$-dependency whatsoever. If one considers large deviations principles one sees a transition from universal to non-universal behavior. Based on the same rescaling as in the Tracy--Widom regime, \eqref{mdp} implies moderate deviations principles with universal rate function $J(s) = \frac{4}{3}s^{3/2}$. In contrast, the rate function~$\eta_V$ of the large deviations principle depends fully on~$V$.

This transition becomes even more elaborate when one considers precise deviations. Again there is no universality in the regime of large deviations. However, in the regime of moderate deviations there is an inf\/inite cascade of regions where the level of $V$-dependency changes. More precisely, the leading order description of $\mathbb{P}_{N,V}( \widetilde{\lambda}_{\max} > s)$ is still universal for $s = o\big(N^{4/15}\big)= o\big(N^{\alpha_0}\big)$. For $N^{\alpha_{k-1}} \leq s \ll N^{\alpha_k}$ the leading order description depends on the $k$-tuple of $V$-dependent numbers $(d_{1, V}, \ldots, d_{k, V})$ and this is how universality fades out as $k \to \infty$, i.e., when approaching the regime of large deviations. This transition can also be observed for the class of repulsive interacting particles introduced in~\cite{GoVe} since the precise deviations results there are similar to ours \cite[Remark~9]{KrVe}. However, in the regime of large deviations the leading order description has not yet been fully understood for those models and a new source of non-universality has been conjectured there \cite[Remark~7]{KrVe}.
\end{Remark}

\subsection{Precise deviations results~II: superlarge deviations}

The task of providing the leading order description for the upper tail of $\lambda_{\max}$ would be fully completed by Corollary~\ref{coro_univ} if we were allowed to choose $q= \infty$ in statement~b)~i). It therefore remains to extend the result of Theorem~\ref{thm_O} to unbounded subsets of $(b_V, \infty)$. In the case of sums of independent variables the corresponding question has been raised under the heading of ``superlarge deviations'' (see, e.g.,~\cite{BorMog}) and we will also use this terminology.

Our second main result states that under additional assumptions on~$V$ the leading order description of the upper tail remains unchanged also in the superlarge regime. In order to formulate our result we introduce new conditions on~$V$ that concern both the size of the region on that $V$ can be extended analytically and the growth of $\operatorname{Re}(V(z))$ as $\operatorname{Re}(z) \to \infty$ on this region.

A function $V$ is said to satisfy ({\bf GA})$_{\infty}$ if (1)--(2) hold:
 \begin{enumerate}\itemsep=0pt
 \item[(1)] $V$ satisf\/ies ({\bf GA}).
\item[(2)] There exists $n\in \mathbb{N}$ and $x_0 > 0$ such that $V$ has an analytic extension on
\begin{gather*}
	\mathcal{U}(n,x_0):=\left\{ z\in\mathbb{C}\,|\, \operatorname{Re}(z) > x_0, \, |\operatorname{Im}(z)| < \frac{1}{(\operatorname{Re}(z))^n} \right\}.
\end{gather*}
Moreover, there exists a constant $d_V>0$ such that for all $z\in\mathcal{U}(n,x_0)$:
\begin{gather*}
	\operatorname{Re}(V(z))\geq d_V \operatorname{Re}(z).
\end{gather*}
\end{enumerate}

Our result on superlarge deviations, which appears to be the f\/irst in the realm of random matrix theory, interacting particle systems and related combinatorial models, reads:

\begin{Theorem}\label{theorem_sl}
Assume that $V$ satisfies $({\bf GA})_\infty$ together with $\frac{V''(x)}{V'(x)^2}=\mathcal{O}(1)$ for $x\to\infty$. Recall the definition of $\mathcal{F}_{N,V}$ in~\eqref{def_FNV}. Then, for sufficiently large values of~$N$,
\begin{gather*}
	\mathbb{P}_{N,V}( \lambda_{\max}>t) = \mathcal{F}_{N,V}(t)
		\left( 1+\mathcal{O}\left( \frac{1}{N}\right) \right),
\end{gather*}
with a uniform $\mathcal{O}$-term for $t \in [b_V+1,\infty)$.
\end{Theorem}

The assumptions that are imposed on $V$ in addition to (\textbf{GA}) are in no way optimal. They are chosen such that at least convex polynomials and in particular the Gaussian unitary ensemble are included. For $V$ that do not satisfy these extra conditions one may try to modify the arguments in the proofs of Lemmas~\ref{K_superlarge} and~\ref{lemma_R}. One sees, e.g., from the arguments around~\eqref{KSSV11} that the faster~$V(x)$ grows for $x \to \infty$ the smaller the domain of analyticity of $V$ needs to be.

\subsection{Overview of the remaining parts of the paper}

The key in proving both of our theorems is that the upper tail probabilities $\mathbb{P}_{N,V} ( \lambda_{\max}>t )=$ $1 - \mathbb{P}_{N,V} ( \lambda_{\max} \leq t )$ are complementary to the gap probabilities that no component of $\lambda$ is contained in the interval $(t, \infty)$. For determinantal ensembles \eqref{defensemble} gap probabilities can be expressed in terms of the kernel~\eqref{defCDK} \cite[Section~4.6]{handbookRM} and one obtains
\begin{gather}\label{darst_O_K}
	\mathbb{P}_{N,V}(\lambda_{\max}>t)=\sum_{k=1}^N \frac{(-1)^{k+1}}{k!} \int_t^{\infty} \cdots \int_t^{\infty} \det \left( K_{N,V}(x_i,x_j)\right) _{1\leq i,j\leq k} \,{\mathrm d}x_1 \cdots{\mathrm d}x_k .
\end{gather}
As it turns out, for all of our analysis just the f\/irst term $\int_t^{\infty} K_{N,V}(x,x)\,{\mathrm d}x$ in the sum of~\eqref{darst_O_K} already
determines the leading order behavior of the tail probabilities. In the situation of moderate and large deviations we show in Section~\ref{sec_proofs} that the asymptotic results on the Christof\/fel--Darboux kernel $K_{N,V}(x,x)$ provided in~\cite{KSSV} together with a~well-known estimate on~$K_{N,V}(x,x)$ for large values of~$x$ suf\/f\/ice to prove Theorem~\ref{thm_O}. The reason why the latter estimate for large $x$ is used in the proof is that the leading order description of the Christof\/fel--Darboux kernel in~\cite{KSSV} is uniform only for $x$ in bounded subsets of $(b_V, \infty)$. In order to treat superlarge deviations uniformity is also required in unbounded subsets of~$(b_V, \infty)$. This is achieved in Section~\ref{sec_proof2} under additional assumptions on~$V$ that are formulated in Theorem~\ref{theorem_sl}. In this situation some of the arguments of~\cite{KSSV} need to be improved which is the content of Appendix~\ref{appendix}.

The results in~\cite{KSSV} have been obtained using the Deift--Zhou \cite{DeiftZhou} nonlinear steepest descent method for Riemann--Hilbert problems, following and improving on previous applications \cite{DKMVZ2, DKMVZ1, KuijlaarsVanlessen, Vanlessen} to orthogonal polynomials and to random matrices.

Finally, we like to mention that we are also able to treat the case that the domain of def\/inition of~$V$ is bounded, but still contains the support of the equilibrium measure $\mu_V$ in its interior. We refer the reader to Remark~\ref{remark_bounded_domain}.

\section{Proof of Theorem \ref{thm_O}}\label{sec_proofs}

As advertised at the end of the Introduction we begin by analyzing the f\/irst summand of \eqref{darst_O_K}.

\begin{Lemma}\label{lemma_M}
Assume that $V$ satisfies $({\bf GA})$ and let $\eta_V$ and $\mathcal{F}_{N,V}$ be given as in \eqref{def1_eta} $($see also~\eqref{def_eta}$)$ and~\eqref{def_FNV}. There exists a number $C > 0$ such that for all
$t \in \big(b_V + CN^{-2/3},\infty\big)$ we have
\begin{gather}\label{relM}
\int_t^{\infty} K_{N,V}(x,x) \,{\mathrm d}x = \mathcal{F}_{N,V}(t) \left( 1+\mathcal{O} \left( \frac{1}{N(t-b_V)^{3/2}}\right)\right).
\end{gather}
The error bound is uniform for $t$ in bounded subsets of $\big(b_V + CN^{-2/3},\infty\big)$.
\end{Lemma}

\begin{proof} Let $S > b_V$ be arbitrary but f\/ixed. We derive \eqref{relM} uniformly for $t \in \big(b_V + CN^{-2/3},S\big]$. We f\/irst show that we only need to consider the integral on a bounded domain. To this end, observe that it follows from~\eqref{def_FNV},~\eqref{def_eta}, and~\eqref{def_G} that there exist positive numbers~$d$, $D$ such that $\mathcal{F}_{N,V}(t) \geq d e^{-ND}$ for all $N$ and all $t \in (b_V ,S]$ (choose, e.g., $D = \eta_V(S) +1$). Next, we use that $K_{N,V}(x,x) = N \rho_N (x)$ where $\rho_N$ denotes the marginal density $\rho_N (x) = \int_{\mathbb{R}^{N-1}} P_{N,V}(x, y) \,{\mathrm d}y$. Well-known estimates from the theory of log-gases (see, e.g., \cite[Theorem~11.1.2]{bookPasturS}, see also \cite[Lemma~5.2]{GoVe}, \cite[Lemma~4.4]{Johansson98}) together with the fact that any $V$ satisfying (\textbf{GA}) grows at least linearly for $x \to \infty$ yield the existence of positive constants~$L$, $\hat{D}$, $\tau$ such that $\rho_N(x) \leq \hat{D} e^{- N \tau x}$ for all $N$ and $x \geq L$. Choosing $M \geq L$ with $M \tau > D$ we see that
\begin{gather}\label{defM}
	\int_M^{\infty} K_{N,V}(x,x)\,{\mathrm d}x =\mathcal{F}_{N,V}(t) \, \mathcal{O}\left( \frac{1}{N}\right),
\end{gather}
uniformly for $t \in (b_V ,S]$. We turn to the remaining part of the integral over the domain $(t, M)$ where we may assume $M > S +1$ without loss of generality. In this domain we now use the information on the integrand provided by \cite[Theorem~1.5(ii)]{KSSV}. Observe that the result in~\cite{KSSV} is stated using a linear rescaling $\lambda_V$ that maps $[-1,1]$ onto~$[a_V,b_V]$: $\lambda_V(t)=\frac{b_V-a_V}{2}t+\frac{b_V+a_V}{2}$. Note in addition that the function $\eta_V$ of the present paper equals $\eta_V \circ \lambda_V^{-1}$ in~\cite{KSSV}. It follows that there exists a positive number $C$ (corresponding to $c^{-1}$ in~\cite{KSSV}) such that
\begin{gather}\label{darst_K}
K_{N,V}(x,x)=\frac{b_V-a_V}{8\pi} \frac{e^{-N\eta_V(x)}}{(x-b_V)(x-a_V)} \left( 1+ \mathcal{O}\left( \frac{1}{N(x-b_V)^{3/2}}\right) \right)
\end{gather}
uniformly for $x \in \big(b_V+CN^{-2/3}, M\big)$, where we have dropped the last $\mathcal{O}$-term in \cite[Theo\-rem~1.5(ii)]{KSSV} since we only consider~$x$ in a bounded set. By~\eqref{defM} and~\eqref{darst_K} it remains to prove
\begin{gather}\label{aux1}
\int_t^M \frac{e^{-N\eta_V(x)}}{(x-b_V)(x-a_V)}\,{\mathrm d}x =\frac{e^{-N\eta_V(t)}}{N(t-b_V)(t-a_V)\eta_V'(t)} \left( 1+ \mathcal{O}\left( \frac{1}{N(t-b_V)^{3/2}}\right)\right).
\end{gather}
We proceed as in \cite[Lemma 4.8]{Diss}. Substituting $u:=\eta_V(x)-\eta_V(t)$ leads to
\begin{gather}\label{aux2}
\int_t^M \frac{e^{-N\eta_V(x)}}{(x-b_V)(x-a_V)}\,{\mathrm d}x =e^{-N\eta_V(t)} \int_0^{\eta_V(M)-\eta_V(t)} f(u) e^{-Nu}\,{\mathrm d}u
\end{gather}
with (observe that $\eta_V$ is strictly monotone, thus invertible)
\begin{gather}\label{aux_f}
	f(u):=\frac{1}{(x(u)-b_V)(x(u)-a_V)\eta_V'(x(u))}, \qquad x(u):=\eta_V^{-1} (u+\eta_V(t)).
\end{gather}
By the mean value theorem there exists a $0 < \xi_u <u$ for every $u \in (0,\eta_V(M)-\eta_V(t)]$ such that $f(u)=f(0)+f'(\xi_u)u$. To estimate $f'$ we use \eqref{def_eta} and obtain
\begin{gather}\label{f'}
\frac{f'(\xi)}{f(\xi)}=- \left[ \frac{3}{2} \left( \frac{1}{x(\xi)-b_V}+\frac{1}{x(\xi)-a_V}\right) +\frac{G_V'(x(\xi))}{G_V(x(\xi))}\right]\frac{1}{\eta_V'(x(\xi))}.
\end{gather}
Since $G_V$ is smooth and strictly positive and $x(\xi) \in (t, M) \subset [b_V, M]$ is contained in a~f\/ixed compact set for all relevant values of $\xi = \xi_u$, we have $G_V'(x(\xi))/G_V(x(\xi))=\mathcal{O}(1)$ and $\eta_V'(x(\xi))^{-1}=\mathcal{O}((x(\xi)-b_V)^{-1/2})
=\mathcal{O}((t-b_V)^{-1/2})$. Moreover, $f(\xi) < f(0)$ and the above mean value representation for $f(u)$ yields
\begin{gather}
	f(u)= f(0) \left( 1+ \mathcal{O} \left( \frac{1}{(t-b_V)^{3/2}}\right) u\right) \nonumber\\
\hphantom{f(u)}{} =
	 \frac{1}{(t-b_V)(t-a_V)\eta_V'(t)}\left( 1+ \mathcal{O} \left( \frac{1}{(t-b_V)^{3/2}}\right) u\right).\label{aux2.5}
\end{gather}
With this representation the integral on the right of~\eqref{aux2} becomes trivial. Recall that we have chosen $M > S+1$. Thus for all $t \in (b_V, S]$ we have
\begin{gather*}
\int_0^{\eta_V(M)-\eta_V(t)} e^{-Nu}\,{\mathrm d}u=\frac{1}{N}\big( 1+\mathcal{O}\big(e^{-cN}\big) \big) \qquad \text{with} \quad c:= \eta_V(M)-\eta_V(S) > 0.
\end{gather*}
As $\int_0^{\infty} u e^{-Nu} \,{\mathrm d}u= N^{-2}$ we have derived \eqref{aux1} and the proof is complete.
\end{proof}

\begin{proof}[Proof of Theorem \ref{thm_O}]
Let $C$ denote the constant introduced in Lemma \ref{lemma_M}. Bounding the probability measure $\mathbb{P}_{N,V}$ by $1$ it follows from \eqref{def_eta} and from the boundedness of $G_V$ that the quotient $\mathbb{P}_{N,V}( \lambda_{\max}>t) / \mathcal{F}_{N,V}(t)$ is a bounded function of $t$ on the interval $\big(b_V, b_V + C N^{-2/3}\big]$. Therefore formula~\eqref{formel_O} holds in that region simply by an appropriate choice for the constant in the $\mathcal{O}$-term.

From now on we may restrict our attention to values $t \in (b_V + C N^{-2/3}, S]$ for some arbitrary but f\/ixed number $S > b_V$. For such values of $t$ we f\/irst record the rough bounds
\begin{gather}\label{aux3}
\int_t^{\infty} K_{N,V}(x,x) \,{\mathrm d}x = \mathcal{O} \left( \frac{1}{N(t-b_V)^{3/2}}\right) = \mathcal{O}(1)
\end{gather}
that follow from Lemma \ref{lemma_M}, \eqref{def_FNV}, \eqref{def_eta} and from the strict positivity and continuity of $G_V$ on $[b_V, S]$.

In order to use \eqref{aux3} for the estimates of the summands in \eqref{darst_O_K} with $k \geq 2$ we recall a basic fact from linear algebra. Suppose that $A = (A_{ij})_{ij}$ is a real, positive def\/inite $k \times k$ matrix. Then the determinant of~$A$ can be estimated by the product of the diagonal entries of~$A$,
\begin{gather}\label{aux4}
 \vert \det A \vert = \det A \leq \prod_{i=1} A_{ii}.
\end{gather}
To see this denote by $B$ a positive def\/inite root of $A = B^2$ and estimate $\det B$ by Hadamard's inequality, i.e., by the product of the (Euclidean) length of the row vectors of~$B$. Then use $\sum\limits_{j=1}^k B_{ij}^2 = \sum\limits_{j=1}^k B_{ij} B_{ji} = A_{ii}$.

It is not dif\/f\/icult to see from \eqref{defCDK} that $\left( K_{N,V}(x_i,x_j)\right) _{1\leq i,j\leq k}$ is a positive def\/inite matrix and we can apply~\eqref{aux4}. Together with Fubini's theorem we arrive at
\begin{gather}\label{aux5}
 \int_t^{\infty} \cdots \int_t^{\infty} \det \left( K_{N,V}(x_i,x_j)\right) _{1\leq i,j\leq k} \,{\mathrm d}x_1 \cdots{\mathrm d}x_k \leq
\left( \int_t^{\infty} K_{N,V}(x,x)\,{\mathrm d}x \right) ^k.
 \end{gather}
Combining \eqref{darst_O_K}, \eqref{aux5}, and \eqref{aux3} gives
\begin{gather*}
		\bigg | \mathbb{P}_{N,V}(\lambda_{\max}>t)-\int_t^{\infty} K_{N,V}(x,x)\,{\mathrm d}x\bigg |
		\leq \bigg ( \int_t^{\infty} K_{N,V}(x,x)\,{\mathrm d}x \bigg ) \mathcal{O} \left( \frac{1}{N(t-b_V)^{3/2}}\right)
	\end{gather*}
and Lemma \ref{lemma_M} completes the proof.
\end{proof}

\begin{Remark}\label{remark_bounded_domain}
Formula \eqref{darst_K}, which is central in the proof of Theorem \ref{thm_O}, is derived in \cite[Theorem 1.5(ii)]{KSSV} in a slightly more general setting. There real analytic functions $V\colon J \to \mathbb{R}$ are considered with $J =[L_-, L_+] \cap \mathbb{R}$ and $-\infty \leq L_- < L_+ \leq \infty$. In addition to conditions (\textbf{GA})(2) and (\textbf{GA})(3) (in the case of inf\/inite $L_-$ resp.~$L_+$) it is assumed that there exist $L_- < a_V < b_V < L_+$ solving equations \eqref{mrs_eq}, which implies that the support of the equilibrium measure is contained in the interior of $J$. In random matrix theory this corresponds to the case of ``soft edges'' (see assumption~\cite[(\textbf{GA})$_1$]{KSSV} and the discussion preceding it).

In the case of $L_+ < \infty$ the tail probabilities $\mathbb{P}_{N,V}(\lambda_{\max}>t)$ are obviously equal to $0$ for $t \geq L_+$.
For $b_V < t < L_+$ the leading order of $\mathbb{P}_{N,V}(\lambda_{\max}>t)$ is provided by the leading order of the integral
\begin{gather*}
\frac{b_V-a_V}{8\pi} \int_t^{L_+} \frac{e^{-N\eta_V(x)}}{(x-b_V)(x-a_V)}\,{\mathrm d}x.
\end{gather*}
A computation shows that $\mathcal{F}_{N,V}(t)$ describes the leading order of $\mathbb{P}_{N,V}(\lambda_{\max}>t)$ if $N(L_+ - t) \to \infty$ for $N \to \infty$. For all this as well as for the results of Theorems~\ref{thm_O} and~\ref{theorem_sl} it is irrelevant whether $L_-$ is f\/inite or inf\/inite (see \cite{Diss} for more details).
\end{Remark}

\section{Proof of Theorem \ref{theorem_sl}}\label{sec_proof2}

In the same way as Theorem \ref{thm_O} followed from Lemma~\ref{lemma_M} the result on superlarge deviations, Theorem~\ref{theorem_sl}, is a consequence of

\begin{Lemma}\label{lemma_sl}
Assume that $V$ satisfies $({\bf GA})_\infty$ together with $\frac{V''(x)}{V'(x)^2}=\mathcal{O}(1)$ for $x\to\infty$. Let $\eta_V$ and $\mathcal{F}_{N,V}$ be given as in~\eqref{def1_eta} $($see also~\eqref{def_eta}$)$ and~\eqref{def_FNV}. Then, for sufficiently large values of~$N$,
\begin{gather*}
\int_t^{\infty} K_{N,V}(x,x) \,{\mathrm d}x = \mathcal{F}_{N,V}(t) \left( 1+\mathcal{O} \left( \frac{1}{N}\right)\right)
\end{gather*}
uniformly for $t \in [b_V + 1,\infty)$.
\end{Lemma}

\begin{proof}
Again we follow the arguments of Section \ref{sec_proofs} but we omit the splitting of the domain of integration in the proof of Lemma~\ref{lemma_M}. This can be done in the following way. We replace~\eqref{darst_K} by
\begin{gather}\label{K_superlarge}
	K_{N,V}(x,x)=\frac{b_V-a_V}{8\pi} \frac{e^{-N\eta_V(x)}}{(x-b_V)(x-a_V)} \left( 1+\mathcal{O} \left(\frac{1}{N}\right)\right)
\end{gather}
uniformly for all $x\geq b_V+1$ and we replace \eqref{aux1} by
\begin{gather}\label{aux1_sl}
\int_t^{\infty} \frac{e^{-N\eta_V(x)}}{(x-b_V)(x-a_V)}\,{\mathrm d}x = \mathcal{F}_{N,V}(t) \left( 1+ \mathcal{O} \left( \frac{1}{N}\right)\right).
\end{gather}
uniformly for all $t \geq b_V+1$. Relation \eqref{K_superlarge} will be proved in Appendix~\ref{appendix} (see Theorem~\ref{theorem_Ksl}, here we need that $N$ is suf\/f\/iciently large) for all~$V$ satisfying $({\bf GA})_\infty$ by adapting the arguments of~\cite{KSSV}. The second claim \eqref{aux1_sl} is the content of the subsequent lemma.
\end{proof}

\begin{Lemma}\label{lemma_integral_super}
Assume that $V$ satisfies $({\bf GA})$ and furthermore $\frac{V''(x)}{V'(x)^2}=\mathcal{O}(1)$ for $x\to\infty$. Then~\eqref{aux1_sl} holds uniformly for all $t \geq b_V+1$.
\end{Lemma}

\begin{proof}
We begin by comparing $\eta_V$ with $V$. Using that the equilibrium measure $\mu_V$ is a probability measure supported on $[a_V, b_V]$ it follows from the def\/inition of $\eta_V$ via~\eqref{def1_eta}, \eqref{aux_eta} that for all $x \geq b_V+1$ we have
\begin{gather}\label{aux6}
\vert \eta_V(x) - V(x) \vert \leq \vert l_V \vert +2 \log(x -a_V), \\ \label{aux7}
\vert \eta_V'(x) - V'(x) \vert \leq \frac{2}{x-b_V}, \qquad \text{and} \qquad
\vert \eta_V''(x) - V''(x) \vert \leq \frac{2}{(x-b_V)^2} .
\end{gather}
For $V$ satisfying (\textbf{GA}) we know that $V(x)$ grows at least linearly as $x \to \infty$. From \eqref{aux6} we then learn that $\eta_V(x) \to \infty$ for $x \to \infty$. Thus the substitution $u=\eta_V(x)-\eta_V(t)$ performed in the proof of Lemma~\ref{lemma_M} gives
\begin{gather}\label{sub_super}
\int_t^{\infty} \frac{e^{-N\eta_V(x)}}{(x-b_V)(x-a_V)}\,{\mathrm d}x = e^{-N\eta_V(t)} \int_0^{\infty} f(u)e^{-Nu}\,{\mathrm d}u
\end{gather}
with $f$ as in~\eqref{aux_f}. In order to derive a bound on $f'$ we use in addition to \eqref{f'} also
\begin{gather*}
\frac{f'(\xi)}{f(\xi)} =- \left[ \left( \frac{1}{x(\xi)-b_V}+\frac{1}{x(\xi)-a_V}\right) +\frac{\eta_V''(x(\xi))}{\eta_V'(x(\xi))}\right]\frac{1}{\eta_V'(x(\xi))}.
\end{gather*}
We are now able to derive for all $u \geq 0$ the representation
\begin{gather}\label{aux8}
f(u)= f(0) ( 1+ \mathcal{O} (1) u)	 = \frac{1}{(t-b_V)(t-a_V)\eta_V'(t)}( 1+ \mathcal{O} (1) u)
\end{gather}
that replaces~\eqref{aux2.5}: Start again with the mean value formula $f(u) = f(0) + f'(\xi_u)u$. Note that for $u > 0$ we have $x(u) \geq x(\xi_u) \geq t \geq b_V+1$. Using the monotonicity of $f$ it suf\/f\/ices to show the boundedness of $f'(\xi)/f(\xi)$ for all $\xi \geq 0$, i.e., $x(\xi) \geq b_V+1$. For $x(\xi)$ in bounded subsets of $[b_V+1, \infty)$ this follows from~\eqref{f'} as in the proof of Lemma~\ref{lemma_M}. In order to treat the case of large values of $x(\xi)$ we f\/irst observe that $V'(b_V) > 0$ (see, e.g., \cite[Lemma~2.1]{Diss}) implying the boundedness of $1/V'(x)$ and of $V''(x)/V'(x)^2$ for $x \geq b_V$ by the assumptions on $V$ (recall in particular (\textbf{GA})(2)). The estimates of~\eqref{aux7} then allow to deduce the boundedness of $1/\eta_V'(x)$ and of $\eta_V''(x)/\eta_V'(x)^2$ for suf\/f\/iciently large values of~$x$.

The statement of Lemma \ref{lemma_integral_super} then follows from \eqref{sub_super}, \eqref{aux8}, and from the trivial identities $\int_0^{\infty} e^{-Nu}\,{\mathrm d}u = N^{-1}$ and $\int_0^{\infty} u e^{-Nu}\,{\mathrm d}u = N^{-2}$.
\end{proof}

\appendix
\section{Appendix}\label{appendix}

The purpose of this appendix is to establish asymptotic formula~\eqref{K_superlarge} that was used in the proof of Theorem~\ref{theorem_sl}. We formulate the corresponding result in Theorem~\ref{theorem_Ksl} which improves on the $\mathcal{O}$-term in \cite[Theorem~1.5(ii)]{KSSV} for unbounded domains under more restrictive assumptions on~$V$ than present in~\cite{KSSV}.

\begin{Theorem}\label{theorem_Ksl}
Assume that $V$ satisfies $({\bf GA})_{\infty}$. Then, for sufficiently large values of~$N$, \eqref{K_superlarge}~holds uniformly for all $x \geq b_V+1$.
\end{Theorem}

\begin{proof}
We begin by recalling the main steps in the proof of \cite[Theorem~1.5(ii)]{KSSV}. First one needs to relate the Christof\/fel--Darboux kernel to the solution of a Riemann--Hilbert problem. Keeping in mind the linear scaling $\lambda_V$ that is used in~\cite{KSSV} (recall our explanation before~\eqref{darst_K}) the f\/irst equation below \cite[equation~(4.7)]{KSSV} reads for $x \neq y$, $x,y \in [b_V + 1, \infty)$:
\begin{gather}\label{KSSV1}
K_{N,V}(x,y) = \frac{1}{x-y} m(y)^T
\begin{pmatrix}
0&-1\\1&0
\end{pmatrix} A^{-1} \hat{R}_+(y)^{-1} \hat{R}_+(x) A m(x),
\end{gather}
where $A$ denotes an invertible complex $2 \times 2$-matrix \cite[equation~(4.3)]{KSSV}, the vector valued function $m$ is def\/ined by $m := k \circ \lambda_V^{-1}$ with $k$ as given in \cite[Theorem 1.3(a), case $x > 1$]{KSSV}, and f\/inally $\hat{R}_+ := R_+ \circ \lambda_V^{-1}$ with the matrix valued function $R_+$ of \cite[Lemma~3.8]{KSSV}. Note that $R_+$ also depends on $N$ which is suppressed in the notation.

Let us f\/irst evaluate $m$. Keeping in mind that $\eta_V$ of the present paper equals $\eta_V\circ \lambda_V^{-1}$ of~\cite{KSSV} and using def\/inition \cite[equation~(1.18)]{KSSV} we compute
\begin{gather}\label{def1_m}
m(x) :=\frac{1}{\sqrt{4\pi}} e^{-\frac{N}{2}\eta_V(x)} \begin{pmatrix}
		c_V(x) \\ c_V(x)^{-1}
	\end{pmatrix},\qquad 	c_V(x):=\frac{(x-b_V)^{1/4}}{(x-a_V)^{1/4}} .
\end{gather}
Before discussing $\hat{R}$ we can already explain how the leading order description of $K_{N, V}(x, x)$ arises. Write
\begin{gather}\label{KSSV2}
\hat{R}_+(y)^{-1} \hat{R}_+(x) = {\operatorname{Id}} + \hat{R}_+(y)^{-1}\big(\hat{R}_+(x) - \hat{R}_+(y)\big) =: (1) + (2) .
\end{gather}
The contribution to $K_{N, V}(x, x)$ from the f\/irst summand $(1)=\operatorname{Id}$, i.e., replace in \eqref{KSSV1} the term $\hat{R}_+(y)^{-1} \hat{R}_+(x)$ by Id and take the limit $y \to x$, is given by
\begin{gather*}
m_1'(x) m_2(x) - m_1(x) m_2'(x) =m_1(x) m_2(x) [(\log m_1)'(x) - (\log m_2)'(x)] \\
\hphantom{m_1'(x) m_2(x) - m_1(x) m_2'(x)}{} = \frac{1}{4\pi} e^{-N \eta_V(x)} \frac{b_V -a_V}{2(x-b_V)(x-a_V)}
\end{gather*}
and hence equals the leading order term in \eqref{K_superlarge}. To estimate the contribution from (2) in~\eqref{KSSV2} we use Lemma~\ref{lemma_R} below that takes the role of \cite[Theorem~3.9]{KSSV} in the proof of \cite[Theorem~1.5(ii)]{KSSV}. Denote by $X_0$ the positive number that is introduced in Lemma~\ref{lemma_R}. State\-ment~(ii) of Lemma~\ref{lemma_R} and the fundamental theorem of calculus provide $\hat{R}_+(x) - \hat{R}_+(y) = \mathcal{O}( \frac{|x-y|}{Nxy})$ for all $x,y > X_0$ with $x \neq y$. Statement~(i) implies that there exists a $X_1 \geq X_0$ such that for all $y > X_1$ the matrix $\hat{R}_+(y)$ is suf\/f\/iciently close to the identity matrix such that $\hat{R}_+(y)^{-1}$ is uniformly bounded for $y > X_1$ and for all~$N$. It follows from~\eqref{def1_m} that $m(x) = e^{-\frac{N}{2}\eta_V(x)} \mathcal{O}(1)$ uniformly for $x > X_0$. The combination of all these estimates shows that the contribution of the second summand~(2) in~\eqref{KSSV2} to $K_{N, V}(x, x)$, again in the limit $y \to x$, is bounded by
\begin{gather*}
e^{-N \eta_V(x)} \mathcal{O}\left( \frac{1}{Nx^2} \right) = \frac{b_V-a_V}{8\pi} \frac{e^{-N\eta_V(x)}}{(x-b_V)(x-a_V)}\,
\mathcal{O}\left( \frac{1}{N} \right),
\end{gather*}
uniformly for $x > X_1$. Thus we are only left to prove \eqref{K_superlarge} uniformly for $x \in [b_V + 1, X_1]$ in case this set is not empty. This, however, follows already from \cite[Theorem~1.5(ii)]{KSSV}.
\end{proof}

We now turn to the analysis of the matrix valued function $\hat{R}_+$ that equals $R_+$ of \cite{KSSV} up to the linear rescaling $\lambda_V$. The function $R$ is analytic on $\mathbb{C} \setminus \Sigma_R$, where $\Sigma_R$ denotes an unbounded contour that is sketched in Fig.~\ref{figure_2} (cf.~\cite[Fig.~2]{KSSV} and observe that the rightmost circle is not present since we are in the case $J = \mathbb{R}$). Note that the role of~$b_V$ is taken by~$1$ since we consider~$R$ instead of~$\hat{R}$. The def\/inition of~$R$ is rather involved \cite[Lemma~3.8]{KSSV} but we do not need it. All that is important for us is that $R$ solves the Riemann--Hilbert problem stated in \cite[Lemma~3.8(i)$R$, (ii)$R$]{KSSV}. The functions~$R_{\pm}$ that appear there are def\/ined on~$\Sigma_R$ as the limits of $R$ when approaching~$\Sigma_R$ from the left resp.\ right with respect to the orientation of~$\Sigma_R$. This f\/inally answers the question how the function $\hat{R}_+$ is def\/ined that appears in~\eqref{KSSV1}.

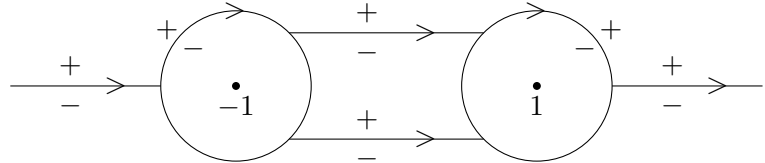
\begin{figure}[ht]\centering
\begin{tikzpicture}[scale=1.0]
	
	\coordinate [label=below: $-1$] (a) at (-2,0);
	\coordinate[label=below: $1$] (b) at (2,0);
	
	\fill (a) circle (1.5pt);
	\fill (b) circle (1.5pt);
	
	\draw (b) circle (1);
	\draw (a) circle (1);
	
	\draw (-5,0)--(-3,0)
		node[pos=0.7] {$>$}
		node[pos=0.4,above] {$+$}
		node[pos=0.4,below] {$-$};
		
	\draw (3,0)--(5,0)
		node[pos=0.7] {$>$}
		node[pos=0.4,above] {$+$}
		node[pos=0.4,below] {$-$};
		
	\draw (-1.2929,0.7071)--(1.2929,0.7071)
		node[pos=0.7] {$>$}
		node[pos=0.4,above] {$+$}
		node[pos=0.4,below] {$-$};	
	
	\draw (-1.2929,-0.7071)--(1.2929,-0.7071)
		node[pos=0.7] {$>$}
		node[pos=0.4,above] {$+$}
		node[pos=0.4,below] {$-$};	
		
	\coordinate [label=0: $>$] (pb) at (-2.3,1);
	\coordinate [label=0: $>$] (pa) at (1.7,1);
	
	\coordinate [label=0: $+$] (plus_a) at (-3.2,0.75);
	\coordinate [label=0: $-$] (plus_a) at (-2.85,0.5);
	
	\coordinate [label=0: $+$] (plus_b) at (2.7,0.75);
	\coordinate [label=0: $-$] (plus_b) at (2.35,0.5);

\end{tikzpicture}
\caption{The contour $\Sigma_R$.}\label{figure_2}
\end{figure}

The next piece of information that we use from \cite[Lemma 3.8]{KSSV} is the smallness of $\Delta_R$ that appears in the jump matrix of the Riemann--Hilbert problem for $R$, i.e., for $\zeta \in \Sigma_R$:
\begin{gather}\label{KSSV3}
R_+(\zeta)=R_-(\zeta) [{\operatorname{Id}} + \Delta_R(\zeta) ] \qquad \text{and} \\ \label{KSSV3a}
\Vert \Delta_R\Vert_{L^1 ( \Sigma_R ) } + \Vert \Delta_R\Vert_{L^{\infty} ( \Sigma_R ) }=\mathcal{O}\big( N^{-1}\big).
\end{gather}
This implies for suf\/f\/iciently large values of $N$ that $R$ has a representation as a Cauchy transform
\begin{gather}\label{KSSV4}
R(z)={\operatorname{Id}}+\frac{1}{2\pi i}\int_{\Sigma_R} \frac{(\Delta_R+\tilde{\mu}\Delta_R)(\zeta)}{\zeta-z}\,{\mathrm d} \zeta, \qquad
	z \in \mathbb{C} \setminus \Sigma_R ,
\end{gather}
where $\tilde{\mu}$ is the solution of a particular singular integral equation (see \cite[Proof of Theorem~3.9, in particular equation~(3.35)]{KSSV}, cf.~\cite[Section~7.5]{Deift} for more background information). From the smallness of~$\Delta_R$ as expressed in~\eqref{KSSV3a} it follows that the underlying singular integral operator is of the form ${\operatorname{Id}} + \mathcal{O}\big( N^{-1}\big)$ as an operator on~$L^2(\Sigma_R)$ and its inverse is thus uniformly bounded for~$N$ suf\/f\/iciently large. This then gives
\begin{gather*}
\Vert \tilde{\mu} \Vert_{L^{2}\left( \Sigma_R \right) }=\mathcal{O}\big( N^{-1}\big)
\end{gather*}
for $N$ suf\/f\/iciently large. In fact, this is the only argument in the proof of Theorem~\ref{theorem_sl} where we use that~$N$ is assumed to be big. For the remaining part of our discussion we assume that~$N$ satisf\/ies this requirement.

There is a dif\/f\/iculty in using \eqref{KSSV4} for our purposes. We are interested in $R_+(x)$ for large values of $x$. Therefore $x \in \Sigma_R$
and the singularity in the denominator of~\eqref{KSSV4} does not allow for pointwise bounds if we only know that the numerator is in some $L^q$ space.

As in \cite{KSSV} we deal with this issue by using the assumed real analyticity of $V$. Then, by~\eqref{def1_eta},~\eqref{aux_eta}, the function~$\eta_V$ is also real analytic. From the def\/inition of $\Delta_R$ \cite[equations~(3.24) and~(3.10)]{KSSV} on the relevant part of $\Sigma_R$, i.e., the rightmost half line in Fig.~\ref{figure_2}, it follows that the jump matrix ${\operatorname{Id}} + \Delta_R$ of the Riemann--Hilbert problem \eqref{KSSV3} also has an analytic extension. In this situation one may deform the contour of the Riemann--Hilbert problem into the region of analyticity of the jump matrix (see, e.g., \cite[Section~7.3]{Deift} for a discussion in a more general setting). Given $x$ large, we use the contour $\Sigma_x$ that is obtained from $\Sigma_R$ by replacing the interval $(x -\kappa_x, x+ \kappa_x)$ by a half circle as shown in Fig.~\ref{figure_1}.

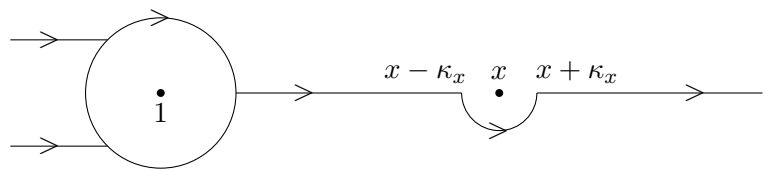
\begin{figure}[ht] \centering
\begin{tikzpicture}[scale=1.0]
	
	\coordinate [label=below: $1$] (b) at (-4,0);
	\coordinate [label=above: $x-\kappa_x\phantom{.........}$] (x-k) at (0,0);	
	\coordinate [label=above: $x$] (x) at (0.5,0.05);
	\coordinate [label=above: $\phantom{..........}x+\kappa_x$] (x+k) at (1,0);

	\fill (b) circle (1.5pt);
	\fill (0.5,0) circle (1.5pt);
	
	
		\draw (b) circle (1);
		\draw (1,0) arc (0:-180:0.5cm);
	
	
		\draw (-3,0) -- (0,0)
			node[pos=0.3] {$>$};
		
		\draw (1,0) -- (4,0)
			node[pos=0.7] {$>$};
	
		\draw (-6,0.7071)--(-4.7071,0.7071)
			node[pos=0.4] {$>$};
		
		\draw (-6,-0.7071)--(-4.7071,-0.7071)
			node[pos=0.4] {$>$};
	
		
		\coordinate [label=0: $>$ ] (pb) at (-4.3,1);
		\coordinate [label=0: $>$] (px) at (0.2,-0.5);

\end{tikzpicture}
\caption{Extract from the contour $\Sigma_x$.}\label{figure_1}
\end{figure}
Of course, $\kappa_x > 0$ has to be chosen such that the lower half disc centered at $x$ and with radius $\kappa_x$ is contained in the domain of analyticity of $\Delta_R$. The jump matrix of the modif\/ied Riemann--Hilbert problem is still given by ${\operatorname{Id}} + \Delta_R$ on the lower half circle and its solution $R_x$ coincides with $R$ except on the lower half disc. Thus we have
\begin{gather}\label{KSSV6}
R_+(x) = R_x(x) \qquad \text{and} \qquad R_+'(x) = R_x'(x) .
\end{gather}
Again we may express the solution $R_x$ of the modif\/ied Riemann--Hilbert problem by a Cauchy transform
\begin{gather}\label{KSSV7}
	R_x(z) -{\operatorname{Id}} = \frac{1}{2\pi i}\int_{\Sigma_x} \frac{(\Delta_R+\tilde{\mu}_x\Delta_R)(\zeta)}{\zeta-z}\,{\mathrm d} \zeta, \qquad
	z \in \mathbb{C} \setminus \Sigma_x .
\end{gather}
To ensure that $\Vert \tilde{\mu}_x \Vert_{L^{2} ( \Sigma_x) }=\mathcal{O} (N^{-1})$ uniformly in $x$, we need to verify that relation~\eqref{KSSV3a} holds uniformly if $\Sigma_R$ is replaced by $\Sigma_x$. Thus we have to estimate~$\Delta_R$ on the lower half circle. This is the point where we begin using the additional assumptions on $V$ that are formulated in $({\bf GA})_\infty$(2). It follows from the def\/inition \cite[equations~(3.24) and (3.10)]{KSSV} that $\Delta_R(z) = \mathcal{O}\big( \vert e^{-N \eta_V(z)} \vert \big)$ for $\lambda_V(z) \in \mathcal{U}(n,x_0)$ and $x_0$ suf\/f\/iciently large. Due to \eqref{aux6} that continues to hold on $\mathcal{U}(n,x_0)$ it follows from the lower bound on $\operatorname{Re}(V)$ formulated in $({\bf GA})_\infty$(2) that also $\operatorname{Re}(\eta_V(z)$) grows at least linearly as $\operatorname{Re}(z)$ becomes large. Hence there exists $\tilde{d} > 0$ such that
\begin{gather}\label{KSSV8}
\Delta_R(z) = \mathcal{O}\big( e^{-N \tilde{d} \operatorname{Re}(z)} \big)
\end{gather}
for $\lambda_V(z) \in \mathcal{U}(n,x_0)$ and $x_0$ suf\/f\/iciently large. This clearly implies~\eqref{KSSV3a} for $\Sigma_x$ with an $\mathcal{O}$-term that is uniform in $x$.

In order to estimate $R_+(x) - {\operatorname{Id}} = R_x(x) - {\operatorname{Id}}$ and $R_+'(x)$ = $R_x'(x)$ (cf.~\eqref{KSSV6}) we recall~\eqref{KSSV7} that also provides a formula for~$R_x'$:
\begin{gather}\label{KSSV9}
	R_x'(z) = \frac{1}{2\pi i}\int_{\Sigma_x} \frac{(\Delta_R+\tilde{\mu}_x\Delta_R)(\zeta)}{(\zeta-z)^2}\,{\mathrm d} \zeta , \qquad
	z \in \mathbb{C} \setminus \Sigma_x .
\end{gather}
For the estimates we split the contour of integration $\Sigma_x$ in two parts
\begin{gather*}
\Sigma_x^{(1)} : = \{ \zeta \in \Sigma_x \colon \vert \zeta - x \vert\ \geq x/2\}
\qquad \text{and} \qquad
\Sigma_x^{(2)} : = \{ \zeta \in \Sigma_x \colon \vert \zeta - x \vert\ < x/2\} .
\end{gather*}
Correspondingly we write $A_j$ resp.~$B_j$, $j \in \{1, 2\}$, for the contributions to the values of $R_x(x) - {\operatorname{Id}}$ resp.\ of~$R_x'(x)$ that stem from integration over $\Sigma_x^{(j)}$ in \eqref{KSSV7} resp.\ in~\eqref{KSSV9}. Since $\Vert \Delta_R \Vert_{L^1(\Sigma_x)}$, $\Vert \tilde{\mu}_x\Vert_{L^2(\Sigma_x)}$, and $\Vert \Delta_R \Vert_{L^2(\Sigma_x)}$ are all of order $1/N$, uniformly for suf\/f\/iciently large $x$, the numerator in~\eqref{KSSV7},~\eqref{KSSV9} is also of order $1/N$ in the $L^1(\Sigma_x)$-norm and it follows immediately that
\begin{gather}\label{KSSV10}
A_1 = \mathcal{O}\left( \frac{1}{Nx} \right) \qquad \text{and} \qquad B_1 = \mathcal{O}\left( \frac{1}{Nx^2} \right)
\end{gather}
for suf\/f\/iciently large values of $x$.

We turn to the contribution from $\Sigma_x^{(2)}$. Due to \eqref{KSSV8} and since the length of $\Sigma_x^{(2)}$ is bounded by $\pi x$ we have that for $x$ large that $\Vert \Delta_R \Vert = \mathcal{O}\big(x e^{-N \tilde{d} x/2}\big)$ in both the $L^1(\Sigma_x^{(2)})$- and
$L^2(\Sigma_x^{(2)})$-norm. Since the distance from $x$ to $\Sigma_x^{(2)}$ is bounded below by the radius $\kappa_x$ we obtain
\begin{gather}\label{KSSV11}
A_2= \mathcal{O}\left( \frac{x e^{-N \tilde{d} x/2}}{\kappa_x} \right) \qquad \text{and} \qquad
B_2 = \mathcal{O}\left( \frac{x e^{-N \tilde{d} x/2}}{\kappa_x^2} \right).
\end{gather}
Assumption $({\bf GA})_\infty$(2) ensures that we may choose $\kappa_x$ of order $x^{-n}$ which suf\/f\/ices amply for proving that~$A_2$ resp.~$B_2$ can be bounded in the same way as $A_1$ resp.~$B_1$ in~\eqref{KSSV10}. In summary, we have derived
\begin{gather*}
R_+(x) -{\operatorname{Id}} = \mathcal{O}\left( \frac{1}{Nx} \right) \qquad \text{and} \qquad R_+'(x) = \mathcal{O}\left( \frac{1}{Nx^2} \right)
\end{gather*}
for $x$ suf\/f\/iciently large. Since $\hat{R}_+ = R_+ \circ \lambda_V^{-1}$ these estimates carry over to $\hat{R}_+$ which is precisely the content of

\begin{Lemma}\label{lemma_R}
Assume that $V$ satisfies $({\bf GA})_\infty$ and let $\hat{R}$ be defined as in the proof of Theorem~{\rm \ref{theorem_Ksl}}.
Then there exists a positive $X_0 > b_V$ such that for sufficiently large values of~$N$
\begin{gather*}
(i) \ \ \hat{R}_+(x) - {\operatorname{Id}} = \mathcal{O}\left( \frac{1}{Nx}\right) \qquad \text{and} \qquad (ii) \ \ \hat{R}_+'(x)=\mathcal{O}\left( \frac{1}{Nx^2}\right)
\end{gather*}
hold uniformly for all $x > X_0$.
\end{Lemma}

\subsection*{Acknowledgements}

All authors acknowledge support received from the Deutsche Forschungsgemeinschaft within the program of the SFB/TR 12. The second author is grateful to Dr.\ Martin Venker for many fruitful discussions while collaborating on~\cite{KrVe} that have inf\/luenced the presentation of the present paper.

\pdfbookmark[1]{References}{ref}
\LastPageEnding

\end{document}